\documentclass[reqno]{amsart}
\usepackage{amssymb}
\usepackage{ifpdf}
\ifpdf
  \usepackage[hyperindex,pagebackref]{hyperref}
\else
  \expandafter\ifx\csname dvipdfm\endcsname\relax
    \usepackage[hypertex,hyperindex,pagebackref]{hyperref}
  \else
    \usepackage[dvipdfm,hyperindex,pagebackref]{hyperref}
  \fi
\fi
\allowdisplaybreaks[4]
\theoremstyle{plain}
\newtheorem{thm}{Theorem}[section]
\newtheorem{open}{Problem}[section]
\newtheorem{lem}{Lemma}[section]
\newtheorem{cor}{Corollary}[section]
\theoremstyle{remark}
\newtheorem{rem}{Remark}[section]
\DeclareMathOperator{\td}{d\mspace{-1mu}}
\numberwithin{equation}{section}

\begin{document}

\title[Eight interesting identities]
{Eight interesting identities involving the exponential function, derivatives, and Stirling numbers of the second kind}

\author[F. Qi]{Feng Qi}
\address{Department of Mathematics, School of Science, Tianjin Polytechnic University, Tianjin City, 300387, China; College of Mathematics, Inner Mongolia University for Nationalities, Tongliao City, Inner Mongolia Autonomous Region, 028043, China}
\email{\href{mailto: F. Qi <qifeng618@gmail.com>}{qifeng618@gmail.com}, \href{mailto: F. Qi <qifeng618@hotmail.com>}{qifeng618@hotmail.com}, \href{mailto: F. Qi <qifeng618@qq.com>}{qifeng618@qq.com}}
\urladdr{\url{http://qifeng618.wordpress.com}}

\begin{abstract}
In the paper, the author establishes some identities which show that the functions $\frac1{(1-e^{\pm t})^k}$ and the derivatives $\bigl(\frac1{e^{\pm t}-1}\bigr)^{(i)}$ can be expressed each other by linear combinations with coefficients involving the combinatorial numbers and the Stirling numbers of the second kind, where $t\ne0$ and $i,k\in\mathbb{N}$.
\end{abstract}



\thanks{This paper was typeset using \AmS-\LaTeX}

\maketitle

\section{Introduction}

Throughout this paper, we denote the set of all positive integers by $\mathbb{N}$.
\par
We observe that
\begin{equation}\label{koumandos-(2.3)}
\frac1{(1-e^{-t})^2}=1+\frac1{e^t-1}-\biggl(\frac1{e^t-1}\biggr)'
\end{equation}
and
\begin{equation}\label{koumandos-1131}
\frac1{(1-e^{-t})^3}=1+\frac1{e^t-1}-\frac32\biggl(\frac1{e^t-1}\biggr)' +\frac12\biggl(\frac1{e^t-1}\biggr)''
\end{equation}
for $t\ne0$. Motivated by these two identities, we naturally pose the following problem.

\begin{open}\label{Qi-exp-identity-open}
For $t\ne0$ and $k\in\mathbb{N}$, determine the numbers $a_{k,i-1}$ for $1\le i\le k$ such that
\begin{equation}\label{open-problem-identity}
\frac1{(1-e^{-t})^k}=1+\sum_{i=1}^{k}a_{k,i-1}\biggl(\frac1{e^t-1}\biggr)^{(i-1)}.
\end{equation}
\end{open}

The first aim of this paper is to find an affirmative answer, which may be stated as Theorem~\ref{answer-open-thm} below, to Problem~\ref{Qi-exp-identity-open}.

\begin{thm}\label{answer-open-thm}
For $i,k\in\mathbb{N}$ with $1\le i\le k$, the coefficients $a_{k,i-1}$ defined in~\eqref{open-problem-identity} may be calculated by
\begin{equation}\label{a-k-M(k-i)-dfn}
  a_{k,i-1}=(-1)^{i^2+1}|M_{k-i+1}(k,i)|,
\end{equation}
where
\begin{align}
|M_1(k,i)|&=\frac1{(i-1)!}\binom{k}{i},\\
|M_j(k,i)|&=
\begin{vmatrix}
\frac1{(i-1)!}\binom{k}{i} & S(i+1,i) & \dotsm  &  S(i+j-1,i)\\
\frac1{i!}\binom{k}{i+1} & S(i+1,i+1) & \dotsm  &  S(i+j-1,i+1)\\
\frac1{(i+1)!}\binom{k}{i+2} & 0 & \dotsm  & S(i+j-1,i+2) \\
\vdots & \vdots & \ddots &\vdots \\
\frac1{(i+j-2)!}\binom{k}{i+j-1} & 0 & \dotsm  & S(i+j-1,i+j-1)
\end{vmatrix}
\end{align}
for $2\le j\le k-i+1$, and
\begin{equation}\label{Stirling-Number-dfn}
S(i,m)=\frac1{m!}\sum_{\ell=0}^m(-1)^{m-\ell}\binom{m}{\ell}\ell^{i}
\end{equation}
for $1\le m\le i$ are the Stirling numbers of the second kind.
\end{thm}

Substituting
\begin{equation}\label{exp-t-exp(-t)-relation}
\frac1{e^t-1}=\frac1{1-e^{-t}}-1
\end{equation}
into the right hand side of~\eqref{open-problem-identity} leads to Corollary~\ref{corollary-identity-cor} below.

\begin{cor}\label{corollary-identity-cor}
For $t\ne0$ and $k\in\mathbb{N}$, we have
\begin{equation}\label{corollary-identity}
\frac1{(1-e^{-t})^k}=\sum_{i=1}^{k}a_{k,i-1}\biggl(\frac1{1-e^{-t}}\biggr)^{(i-1)},
\end{equation}
where $a_{k,i-1}$ is determined by~\eqref{a-k-M(k-i)-dfn}.
\end{cor}

Theorem~\ref{answer-open-thm} and Corollary~\ref{corollary-identity-cor} tell us that the function $\frac1{(1-e^{-t})^k}$ can be expressed by linear combinations of the derivatives $\bigl(\frac1{e^t-1}\bigr)^{(i-1)}$ and $\bigl(\frac1{1-e^{-t}}\bigr)^{(i-1)}$ respectively.
\par
Combining the equality~\eqref{exp-t-exp(-t)-relation} with the left hand sides of the equalities~\eqref{open-problem-identity} and~\eqref{corollary-identity} and using the binomial theorem imply that the function $\frac1{(e^{t}-1)^k}$ can also be represented by some linear combinations of the derivatives $\bigl(\frac1{e^t-1}\bigr)^{(i-1)}$ and $\bigl(\frac1{1-e^{-t}}\bigr)^{(i-1)}$ respectively. To discover coefficients in these linear combinations becomes the second aim of this paper.

\begin{thm}\label{answer-open-thm=2}
For $i,k\in\mathbb{N}$ with $1\le i\le k$, the identity
\begin{equation}\label{open-problem-identity-b}
\frac1{(e^{t}-1)^k}=\sum_{i=1}^{k}b_{k,i-1}\biggl(\frac1{e^t-1}\biggr)^{(i-1)}
\end{equation}
is valid and the coefficients $b_{k,i-1}$ can be computed by
\begin{equation}\label{a-k-M(k-i)-dfn-b}
b_{k,i-1}=(-1)^{i-1}a_{k,i-1},
\end{equation}
where $a_{k,i-1}$ is just the quantity~\eqref{a-k-M(k-i)-dfn}.
\end{thm}

Utilizing~\eqref{exp-t-exp(-t)-relation} in the right hand side of~\eqref{open-problem-identity-b} brings about Corollary~\ref{answer-open-cor=2} below.

\begin{cor}\label{answer-open-cor=2}
For $i,k\in\mathbb{N}$ with $1\le i\le k$, the identity
\begin{equation}\label{open-cor-identity-b}
\frac1{(e^{t}-1)^k}=1+\sum_{i=1}^{k}b_{k,i-1}\biggl(\frac1{1-e^{-t}}\biggr)^{(i-1)}
\end{equation}
validates, where the coefficients $b_{k,i-1}$ are decided by~\eqref{a-k-M(k-i)-dfn-b}.
\end{cor}

\section{Lemmas}

For proving Theorems~\ref{answer-open-thm} and~\ref{answer-open-thm=2}, we need the following lemmas.

\begin{lem}\label{exp-deriv-exp-lem}
For $i\in\mathbb{N}\cup\{0\}$, we have
\begin{equation}\label{exp-deriv-exp}
\biggl(\frac1{e^t-1}\biggr)^{(i)}=\sum_{m=1}^{i+1}\frac{\lambda_{i,m}}{(e^t-1)^m},
\end{equation}
where
\begin{equation}\label{lambda-stirling-relation}
\lambda_{i,m}=(-1)^i(m-1)!S(i+1,m).
\end{equation}
\end{lem}

\begin{proof}
We prove this lemma by induction on $i$.
\par
For $i=0,1$, it is simple to verify the identity~\eqref{exp-deriv-exp}.
\par
Assume that the identity~\eqref{exp-deriv-exp} is valid for some $i>1$.
\par
By the inductive hypothesis and a direct differentiation on both sides of the identity~\eqref{exp-deriv-exp}, we obtain
\begin{align*}
\biggl(\frac1{e^t-1}\biggr)^{(i+1)}&=\sum_{k=1}^{i+1}\lambda_{i,k}\frac{\td}{\td t} \biggl[\frac1{(e^t-1)^k}\biggr]\\
&=-\sum_{k=1}^{i+1}k\lambda_{i,k}\biggl[\frac{1}{(e^t-1)^k}+\frac{1}{(e^t-1)^{k+1}}\biggr]\\
&=-\Biggl[\sum_{k=1}^{i+1}\frac{k\lambda_{i,k}}{(e^t-1)^k} +\sum_{k=1}^{i+1}\frac{k\lambda_{i,k}}{(e^t-1)^{k+1}}\Biggr]\\
&=-\Biggl[\sum_{k=1}^{i+1}\frac{k\lambda_{i,k}}{(e^t-1)^k} +\sum_{k=2}^{i+2}\frac{(k-1)\lambda_{i,k-1}}{(e^t-1)^{k}}\Biggr]\\
&=-\Biggl[\frac{\lambda_{i,1}}{e^t-1} +\sum_{k=2}^{i+1}\frac{k\lambda_{i,k} +(k-1)\lambda_{i,k-1}}{(e^t-1)^{k}} +\frac{(i+1)\lambda_{i,i+1}}{(e^t-1)^{i+2}}\Biggr]
\end{align*}
and
\begin{align*}
k\lambda_{i,k} +(k-1)\lambda_{i,k-1}&=\sum_{\ell=1}^{k-1}(-1)^{i+k+\ell}\biggl[k\binom{k-1}{\ell-1} -(k-1)\binom{k-2}{\ell-1}\biggr]\ell^i+(-1)^ik^{i+1}\\
&=\sum_{\ell=1}^{k-1}(-1)^{i+k+\ell}\binom{k-1}{\ell-1}\ell^{i+1}+(-1)^ik^{i+1}\\
&=(-1)^i\sum_{\ell=1}^{k}(-1)^{k+\ell}\binom{k-1}{\ell-1}\ell^{i+1}\\
&=(-1)^{i+2}(k-1)!S(i+2,k)\\
&=-\lambda_{i+1,k}
\end{align*}
for $2\le k\le i+1$. Moreover, we observe that
\begin{equation*}
\lambda_{i,1}=-\lambda_{i+1,1}=(-1)^i \quad\text{and}\quad (i+1)\lambda_{i,i+1}=-\lambda_{i+1,i+2}=(-1)^i(i+1)!.
\end{equation*}
By induction on $i$, the identity~\eqref{exp-deriv-exp} is thus proved.
\end{proof}

Using the equality~\eqref{exp-t-exp(-t)-relation} in the left hand side of~\eqref{exp-deriv-exp} leads to Corollary~\ref{exp-deriv-exp-lem-cor} below.

\begin{cor}\label{exp-deriv-exp-lem-cor}
For $i\in\mathbb{N}$, we have
\begin{equation}\label{exp-deriv-exp(-t)}
\biggl(\frac1{1-e^{-t}}\biggr)^{(i)}=\sum_{m=1}^{i+1}\frac{\lambda_{i,m}}{(e^t-1)^m},
\end{equation}
where $\lambda_{i,m}$ is defined by~\eqref{lambda-stirling-relation}.
\end{cor}

\begin{lem}\label{exp-deriv-exp-lem-neg-t}
For $i\in\mathbb{N}\cup\{0\}$, we have
\begin{equation}\label{exp-deriv-exp-mu}
\biggl(\frac1{1-e^{-t}}\biggr)^{(i)}=\sum_{m=1}^{i+1}\frac{\mu_{i,m}}{(1-e^{-t})^m},
\end{equation}
where
\begin{equation}\label{mu-stirling-relation}
\mu_{i,m}=(-1)^{m+1}(m-1)!S(i+1,m).
\end{equation}
\end{lem}

\begin{proof}
A simple differentiation yields
\begin{align*}
\biggl(\frac1{1-e^{-t}}\biggr)^{(i+1)}&=\sum_{k=1}^{i+1}\mu_{i,k}\frac{\td}{\td t} \biggl[\frac1{(1-e^{-t})^k}\biggr]\\
&=\sum_{k=1}^{i+1}k\mu_{i,k}\biggl[\frac{1}{(1-e^{-t})^k}-\frac{1}{(1-e^{-t})^{k+1}}\biggr]\\
&=\sum_{k=1}^{i+1}\frac{k\mu_{i,k}}{(1-e^{-t})^k} -\sum_{k=1}^{i+1}\frac{k\mu_{i,k}}{(1-e^{-t})^{k+1}}\\
&=\sum_{k=1}^{i+1}\frac{k\mu_{i,k}}{(1-e^{-t})^k} -\sum_{k=2}^{i+2}\frac{(k-1)\mu_{i,k-1}}{(1-e^{-t})^{k}}\\
&=\frac{\mu_{i,1}}{1-e^{-t}} +\sum_{k=2}^{i+1}\frac{k\mu_{i,k} -(k-1)\mu_{i,k-1}}{(1-e^{-t})^{k}} -\frac{(i+1)\mu_{i,i+1}}{(1-e^{-t})^{i+2}}.
\end{align*}
Equating coefficients of $\frac1{(1-e^{-t})^{k}}$ for $1\le k\le i+2$ in
\begin{equation*}
\frac{\mu_{i,1}}{1-e^{-t}} +\sum_{k=2}^{i+1}\frac{k\mu_{i,k} -(k-1)\mu_{i,k-1}}{(1-e^{-t})^{k}} -\frac{(i+1)\mu_{i,i+1}}{(1-e^{-t})^{i+2}} =\sum_{k=1}^{i+2}\frac{\mu_{i+1,k}}{(1-e^{-t})^k}
\end{equation*}
gives
\begin{equation}
\begin{split}\label{recursion=i=i+1}
\mu_{i+1,1}&=\mu_{i,1},\\
\mu_{i+1,i+2}&=-(i+1)\mu_{i,i+1},
\end{split}
\end{equation}
and
\begin{equation}\label{i+1-k-i-k-i-k-1-recurs}
\mu_{i+1,k}=k\mu_{i,k} -(k-1)\mu_{i,k-1}
\end{equation}
for $1\le k\le i+1$.
\par
From an obvious fact that $\mu_{0,1}=1$ and the recursion formulas in~\eqref{recursion=i=i+1}, we can easily deduce
\begin{equation}\label{i-1-i-i+1-recurs}
\mu_{i,1}=1
\end{equation}
and
\begin{equation}\label{i-1-i-i+1-recurs-II}
\mu_{i,i+1}=(-1)^ii!
\end{equation}
for $i\ge0$.
\par
Employing equalities~\eqref{i-1-i-i+1-recurs} and~\eqref{i-1-i-i+1-recurs-II} and recurring repeatedly the formula~\eqref{i+1-k-i-k-i-k-1-recurs} result in
\begin{equation}\label{mu-lambda-relation}
\mu_{i,k}=(-1)^{i+k+1}\lambda_{i,k}.
\end{equation}
Replacing~\eqref{lambda-stirling-relation} into the above equality leads to~\eqref{mu-stirling-relation}. The proof of Lemma~\ref{exp-deriv-exp-lem-neg-t} is thus completed.
\end{proof}

Using the equality~\eqref{exp-t-exp(-t)-relation} in the left hand side of~\eqref{exp-deriv-exp-mu} brings out Corollary~\ref{exp-deriv-exp-lem-cor-mu} below.

\begin{cor}\label{exp-deriv-exp-lem-cor-mu}
For $i\in\mathbb{N}$, we have
\begin{equation}\label{exp-deriv-exp-mu(-t)}
\biggl(\frac1{e^{t}-1}\biggr)^{(i)}=\sum_{m=1}^{i+1}\frac{\mu_{i,m}}{(1-e^{-t})^m},
\end{equation}
where $\mu_{i,m}$ is defined by~\eqref{mu-stirling-relation}.
\end{cor}

Lemmas~\ref{exp-deriv-exp-lem} and~\ref{exp-deriv-exp-lem-neg-t} together with Corollaries~\ref{exp-deriv-exp-lem-cor} and~\ref{exp-deriv-exp-lem-cor-mu} declare that the derivatives $\bigl(\frac1{e^{t}-1}\bigr)^{(i)}$ and $\bigl(\frac1{1-e^{-t}}\bigr)^{(i)}$ can be expressed by linear combinations of the functions $\frac1{(e^t-1)^m}$ and $\frac1{(1-e^{-t})^m}$ respectively.

\section{Proofs of Theorems~\ref{answer-open-thm} and~\ref{answer-open-thm=2}}

Now it is time for us to prove Theorems~\ref{answer-open-thm} and~\ref{answer-open-thm=2}.

\begin{proof}[Proof of Theorem~\ref{answer-open-thm}]
It is easy to see that
\begin{equation*}
\frac1{(1-e^{-t})^k}=\biggl(1+\frac1{e^t-1}\biggr)^k=\sum_{m=0}^{k}\binom{k}{m}\frac1{(e^t-1)^m}.
\end{equation*}
On the other hand, by virtue of~\eqref{exp-deriv-exp} and changing the order of summations, it follows that
\begin{align*}
\sum_{i=0}^{k-1}a_{k,i}\biggl(\frac1{e^t-1}\biggr)^{(i)}& =\sum_{i=0}^{k-1}a_{k,i}\sum_{m=1}^{i+1}\frac{\lambda_{i,m}}{(e^t-1)^m}
=\sum_{m=1}^k\Biggl(\sum_{i=m-1}^{k-1} \lambda_{i,m}a_{k,i}\Biggr)\frac1{(e^t-1)^m}.
\end{align*}
Equating the coefficients of $\frac1{(e^t-1)^m}$ in
\begin{equation*}
\sum_{m=0}^{k}\binom{k}{m}\frac1{(e^t-1)^m} =1+\sum_{m=1}^k\Biggl(\sum_{i=m-1}^{k-1} \lambda_{i,m}a_{k,i}\Biggr)\frac1{(e^t-1)^m}
\end{equation*}
yields a system of linear equations
\begin{equation}\label{equ-sys-lambda}
\sum_{i=m-1}^{k-1}\lambda_{i,m}a_{k,i}=\binom{k}{m},\quad 1\le m\le k.
\end{equation}
By Cramer's rule in linear algebra, it is procured that
\begin{equation*}
a_{k,i-1}=\frac{|\Lambda_{k,i}|}{|\Lambda_k|}, \quad 1\le i\le k,
\end{equation*}
where
\begin{align}
|\Lambda_k|&=
\begin{vmatrix}\label{Lambda-matrix-dfn}
\lambda_{0,1} & \lambda_{1,1} & \lambda_{2,1} & \dotsm  & \lambda_{k-2,1} & \lambda_{k-1,1} \\
0 & \lambda_{1,2} & \lambda_{2,2} & \dotsm  & \lambda_{k-2,2} & \lambda_{k-1,2} \\
0& 0 & \lambda_{2,3} & \dotsm  & \lambda_{k-2,3} & \lambda_{k-1,3} \\
\vdots & \vdots & \vdots & \ddots & \vdots & \vdots  \\
0 & 0 & 0 & \dotsm  & \lambda_{k-2,k-1} & \lambda_{k-1,k-1} \\
0 & 0 & 0 & \dotsm  & 0 & \lambda_{k-1,k}
\end{vmatrix}\\
&=\prod_{\ell=1}^{k}\lambda_{\ell-1,\ell}, \notag\\
|\Lambda_{k,i}|&=
\begin{vmatrix}
\lambda_{0,1} &  \dotsm  & \lambda_{i-2,1} & \binom{k}{1} & \lambda_{i,1} & \dotsm & \lambda_{k-1,1} \\
0 &  \dotsm  & \lambda_{i-2,2} & \binom{k}{2} & \lambda_{i,2} & \dotsm  & \lambda_{k-1,2} \\
\vdots & \ddots & \vdots &\vdots & \vdots & \ddots & \vdots \\
0 & \dotsm  & \lambda_{i-2,i-1} & \binom{k}{i-1} & \lambda_{i,i-1} & \dotsm  &  \lambda_{k-1,i-1} \\
0 & \dotsm  & 0 & \binom{k}{i} & \lambda_{i,i} & \dotsm  &  \lambda_{k-1,i} \\
0 & \dotsm  & 0 & \binom{k}{i+1} & \lambda_{i,i+1} & \dotsm  &  \lambda_{k-1,i+1} \\
\vdots & \ddots & \vdots &\vdots & \vdots & \ddots &\vdots  \\
0 & \dotsm  & 0 & \binom{k}{k-1} & 0 & \dotsm  &   \lambda_{k-1,k-1} \\
0 & \dotsm  & 0 & \binom{k}{k} & 0 & \dotsm  & \lambda_{k-1,k}
\end{vmatrix}\\
&=\prod_{\ell=1}^{i-1}\lambda_{\ell-1,\ell}
\begin{vmatrix}
\binom{k}{i} & \lambda_{i,i} & \dotsm  &  \lambda_{k-2,i}  & \lambda_{k-1,i} \\
\binom{k}{i+1} & \lambda_{i,i+1} & \dotsm  &  \lambda_{k-2,i+1}  & \lambda_{k-1,i+1}\\
\binom{k}{i+2} & 0 & \dotsm  &  \lambda_{k-2,i+2}  & \lambda_{k-1,i+2} \\
\vdots & \vdots & \ddots &\vdots & \vdots \\
\binom{k}{k-1} & 0 & \dotsm  &  \lambda_{k-2,k-1}  & \lambda_{k-1,k-1} \\
\binom{k}{k} & 0 & \dotsm  & 0 & \lambda_{k-1,k}
\end{vmatrix}\notag\\
&=\prod_{\ell=1}^{i-1}\lambda_{\ell-1,\ell}|D_{k-i+1}(k,i)|,\notag
\end{align}
and
\begin{align}
  |D_j(k,i)|&=
  \begin{cases}
  \binom{k}{i}, & j=1\\
\begin{vmatrix}
\binom{k}{i} & \lambda_{i,i} & \dotsm  &  \lambda_{i+j-2,i}\\
\binom{k}{i+1} & \lambda_{i,i+1} & \dotsm  &  \lambda_{i+j-2,i+1}\\
\binom{k}{i+2} & 0 & \dotsm  &  \lambda_{i+j-2,i+2} \\
\vdots & \vdots & \ddots &\vdots \\
\binom{k}{i+j-1} & 0 & \dotsm  &  \lambda_{i+j-2,i+j-1}
\end{vmatrix}, & 2\le j\le k-i+1
\end{cases}\\
&=(-1)^{(2i+j-2)(j-1)/2}\prod_{\ell=1}^{j}(i+\ell-2)!|M_j(k,i)|. \notag
\end{align}
As a result,
\begin{align*}
a_{k,i-1}&=\frac{|D_{k-i+1}(k,i)|}{\prod_{\ell=i}^k\lambda_{\ell-1,\ell}}\\
&=\frac{(-1)^{(k-i)(k+i-1)/2}\prod_{\ell=1}^{k-i+1}(i+\ell-2)!|M_{k-i+1}(k,i)|} {(-1)^{(k+i-2)(k-i+1)/2}\prod_{\ell=i}^k(\ell-1)!}\\
&=(-1)^{k(k+1)-(i-1)^2}|M_{k-i+1}(k,i)|\\
&=(-1)^{i^2+1}|M_{k-i+1}(k,i)|.
\end{align*}
The proof of Theorem~\ref{answer-open-thm} is complete.
\end{proof}

\begin{proof}[Proof of Theorem~\ref{answer-open-thm=2}]
Making use of~\eqref{exp-t-exp(-t)-relation} and the binomial theorem, the left hand side of~\eqref{open-problem-identity-b} becomes
\begin{equation*}
\frac1{(e^{t}-1)^k}=\biggl(\frac1{1-e^{-t}}-1\biggr)^k =\sum_{\ell=0}^k(-1)^{k+\ell}\binom{k}{\ell}\frac1{(1-e^{-t})^{\ell}}.
\end{equation*}
Substituting~\eqref{exp-deriv-exp-mu(-t)} into the right hand side of~\eqref{open-problem-identity-b} and interchanging the order of summations generate
\begin{multline*}
\sum_{i=1}^{k}b_{k,i-1}\biggl(\frac1{e^t-1}\biggr)^{(i-1)}
=b_{k,0}\biggl(\frac1{1-e^{-t}}-1\biggr) +\sum_{i=2}^{k}b_{k,i-1} \sum_{m=1}^{i}\frac{\mu_{i-1,m}}{(1-e^{-t})^m}\\
\begin{aligned}
&=-b_{k,0}+\Biggl(b_{k,0}+\sum_{i=2}^k\mu_{i-1,1}b_{k,i-1}\Biggr)\frac{1}{1-e^{-t}} +\sum_{m=2}^k \Biggl(\sum_{i=m}^k \mu_{i-1,m}b_{k,i-1}\Biggr)\frac{1}{(1-e^{-t})^m}\\
&=-b_{k,0}+\sum_{m=1}^k \Biggl(\sum_{i=m}^k \mu_{i-1,m}b_{k,i-1}\Biggr)\frac{1}{(1-e^{-t})^m}.
\end{aligned}
\end{multline*}
Equating coefficients of $\frac1{(1-e^{-t})^{\ell}}$ in
\begin{equation*}
\sum_{\ell=0}^k(-1)^\ell\binom{k}{\ell}\frac1{(1-e^{-t})^{\ell}}
=-b_{k,0}+\sum_{m=1}^k \Biggl(\sum_{i=m}^k \mu_{i-1,m}b_{k,i-1}\Biggr)\frac{1}{(1-e^{-t})^m}
\end{equation*}
produces
\begin{equation}\label{b(k,0)-1}
b_{k,0}=-1
\end{equation}
and
\begin{equation}\label{b(k,0)-3}
\sum_{i=\ell}^k \mu_{i-1,\ell}b_{k,i-1}=(-1)^\ell\binom{k}{\ell},\quad 1\le\ell\le k.
\end{equation}
By virtue of~\eqref{mu-lambda-relation}, the equations system~\eqref{b(k,0)-3} can be rearranged as
\begin{equation}\label{b(k,0)-30-rewrit}
\sum_{i=\ell}^k(-1)^{i-1} \lambda_{i-1,\ell}b_{k,i-1}=\binom{k}{\ell},\quad 1\le\ell\le k.
\end{equation}
Comparing~\eqref{b(k,0)-30-rewrit} with~\eqref{equ-sys-lambda} reveals that
\begin{equation}
b_{k,i}=(-1)^ia_{k,i},\quad 0\le i\le k-1.
\end{equation}
The proof of Theorem~\ref{answer-open-thm=2} is thus complete.
\end{proof}

\section{Remarks}

In this section, we further list several remarks on the above two theorems, two lemmas, and four corollaries.

\begin{rem}
The identities obtained in Theorem~\ref{answer-open-thm}, Corollary~\ref{corollary-identity-cor}, Corollary~\ref{answer-open-cor=2}, and Theorem~\ref{answer-open-thm=2} imply that the linear ordinary differential equations
\begin{equation}\label{open-problem-ODE}
\sum_{i=1}^{k}a_{k,i-1}y^{(i-1)}=F_n(t)
\end{equation}
and
\begin{equation}
\sum_{i=1}^{k}(-1)^{i-1}a_{k,i-1}y^{(i-1)}=G_n(t)
\end{equation}
for
\begin{equation}
F_n(t)=
\begin{cases}
\dfrac1{(1-e^{-t})^k}-1, & n=1\\
\dfrac1{(1-e^{-t})^k}, & n=2\\
\end{cases}
\end{equation}
and
\begin{equation}
G_n(t)=
\begin{cases}
\dfrac1{(e^{t}-1)^k}-1, & n=1\\
\dfrac1{(e^{t}-1)^k}, & n=2
\end{cases}
\end{equation}
have unique solutions
\begin{equation}
\frac1{e^t-1}, \quad \frac1{1-e^{-t}}, \quad \frac1{1-e^{-t}}, \quad\text{and}\quad \frac1{e^t-1}
\end{equation}
respectively.
\end{rem}

\begin{rem}
The equality~\eqref{exp-deriv-exp} can be rewritten as
\begin{equation*}
\begin{pmatrix}
  \frac1{e^t-1} \\
  \bigl(\frac1{e^t-1}\bigr)' \\
  \bigl(\frac1{e^t-1}\bigr)'' \\
   \vdots                             \\
  \bigl(\frac1{e^t-1}\bigr)^{(i-1)} \\
  \bigl(\frac1{e^t-1}\bigr)^{(i)}
\end{pmatrix}
=\Lambda_{i+1}^{T}
\begin{pmatrix}
  \frac{1}{e^t-1} \\
  \frac{1}{(e^t-1)^2} \\
  \frac{1}{(e^t-1)^3} \\
  \vdots \\
  \frac{1}{(e^t-1)^i}\\
  \frac{1}{(e^t-1)^{i+1}}
\end{pmatrix},
\end{equation*}
where $\Lambda_{i+1}^{T}$ denotes the transpose of the matrix $\Lambda_{i+1}$ defined by~\eqref{Lambda-matrix-dfn}.
Therefore,
\begin{equation*}
\begin{pmatrix}
  \frac{1}{e^t-1} \\
  \frac{1}{(e^t-1)^2} \\
  \frac{1}{(e^t-1)^3} \\
  \vdots \\
  \frac{1}{(e^t-1)^i}\\
  \frac{1}{(e^t-1)^{i+1}}
\end{pmatrix}
=\bigl(\Lambda_{i+1}^{T}\bigr)^{-1}
\begin{pmatrix}
  \frac1{e^t-1} \\
  \bigl(\frac1{e^t-1}\bigr)' \\
  \bigl(\frac1{e^t-1}\bigr)'' \\
   \vdots                             \\
  \bigl(\frac1{e^t-1}\bigr)^{(i-1)} \\
  \bigl(\frac1{e^t-1}\bigr)^{(i)}
\end{pmatrix},
\end{equation*}
where $\Lambda_{i+1}^{-1}$ stands for the inverse of $\Lambda_{i+1}$ respectively. On the other hand, Theorem~\ref{answer-open-thm=2} can be rearranged directly as
\begin{equation*}
\begin{pmatrix}
  \frac{1}{e^t-1} \\
  \frac{1}{(e^t-1)^2} \\
  \frac{1}{(e^t-1)^3} \\
  \vdots \\
  \frac{1}{(e^t-1)^k}
\end{pmatrix}
=(-1)^{(k-1)k/2}
\begin{pmatrix}
   a_{1,0} & 0 & 0 & \dotsm & 0 \\
   a_{2,0} &  a_{2,1} & 0 & \dotsm & 0 \\
   a_{3,0} &  a_{3,1} &  a_{3,2} & \dotsm & 0\\
  \vdots & \vdots & \vdots & \ddots & \vdots \\
   a_{k,0} &  a_{k,1} &  a_{k,2} & \dotsm &  a_{k,k-1}
\end{pmatrix}
\begin{pmatrix}
  \frac1{e^t-1} \\
  \bigl(\frac1{e^t-1}\bigr)' \\
  \bigl(\frac1{e^t-1}\bigr)'' \\
   \vdots                     \\
  \bigl(\frac1{e^t-1}\bigr)^{(k-1)}
\end{pmatrix}.
\end{equation*}
As a result, it follows that
\begin{equation}
\Lambda_{k}^{-1} =(-1)^{(k-1)k/2}
\begin{pmatrix}
   a_{1,0} & a_{2,0} & a_{3,0} & \dotsm & a_{k,0} \\
   0 &  a_{2,1} & a_{3,1} & \dotsm & a_{k,1} \\
   0 &  0 &  a_{3,2} & \dotsm & a_{k,2}\\
  \vdots & \vdots & \vdots & \ddots & \vdots \\
   0 &  0 &  0 & \dotsm &  a_{k,k-1}
\end{pmatrix}.
\end{equation}
This reveals that Theorem~\ref{answer-open-thm=2} and Lemma~\ref{exp-deriv-exp-lem} are essentially equivalent to each other.
\end{rem}

\begin{rem}
The identities just-obtained in this paper show us that the derivatives $\bigl(\frac1{e^{t}-1}\bigr)^{(i)}$ and $\bigl(\frac1{1-e^{-t}}\bigr)^{(i)}$ and the functions $\frac1{(e^t-1)^m}$ and $\frac1{(1-e^{-t})^m}$ can be expressed each other by linear combinations with coefficients involving the combinatorial numbers and the Stirling numbers of the second kind. This means that all the theorems, lemmas, and corollaries acquired in this paper are equivalent to each other.
\end{rem}

\end{document}